\documentclass[12pt]{amsart}
\usepackage{amssymb,amsmath,amsthm,bm}
\usepackage[colorinlistoftodos,prependcaption,textsize=tiny]{todonotes}
 \definecolor{darkblue}{RGB}{0,0,160}
\usepackage{fouriernc} 
\usepackage[colorlinks=true,allcolors=darkblue]{hyperref}
\DeclareSymbolFont{usualmathcal}{OMS}{cmsy}{m}{n}
\DeclareSymbolFontAlphabet{\mathcal}{usualmathcal}

\usepackage[T1]{fontenc}

\usepackage{color}

\usepackage{parskip}
\usepackage{setspace}

\usepackage{graphicx}

\newtheorem{theorem}{Theorem}[section]
\newtheorem{remark}{Remark}[section]

\newtheorem{lemma}[theorem]{Lemma}

\newtheorem*{definition*}{Definition}

\begin{document}
\title{The Erd\H{o}s-Falconer distance problem in the tree setting}
\author{Thang Pham\and Steven Senger\and Dung The Tran}

  \address{Hanoi University of Science, Vietnam National University}
\email{thangpham.math@vnu.edu.vn}

\address{Department of Mathematics, Missouri State University}
\email{StevenSenger@MissouriState.edu}

\address{Department of Mathematics, National Taiwan Normal University}
\email{tranthedung56@gmail.com}

\maketitle  
\begin{abstract}
The recent breakthrough of Guth, Iosevich, Ou, and Wang (2019) on the Falconer distance problem states that for a compact set $A\subset \mathbb{R}^2$, if the Hausdorff dimension of $A$ is greater than $\frac{5}{4}$, then the distance set $\Delta(A)$ has positive Lebesgue measure. In a very recent paper, Murphy, Petridis, Pham, Rudnev, and Stevens (2022) proved the prime field version of this result, namely, for $E\subset\mathbb{F}_p^2$ with $|E|\gg p^{5/4}$, there exist many points $x\in E$ such that the number of distinct distances from $x$ is at least $cp$. The main purpose of this paper is to provide extensions  in a very general structure of pinned trees, which is inspired by the recent work due to Ou and Taylor (2021).
\end{abstract}
\section{Introduction}
For $A\subset \mathbb{R}^d$, we define its distance set by
\[\Delta(A):=\{|x-y|\colon x, y\in A\}.\]
The celebrated Falconer distance conjecture says that for any compact set $A\subset \mathbb{R}^d$, if its Hausdorff dimension is greater than $\frac{d}{2}$, then the distance set has positive Lebesgue measure, i.e. $\mathcal{L}^1(\Delta(A))>0$. In a recent breakthrough paper, Guth, Iosevich, Ou, and Wang \cite{alex-fal} proved that for a compact set $A\subset \mathbb{R}^2$, if $\dim_H(A)>\frac{5}{4}$, then $\mathcal{L}^1(\Delta(A))>0$. This result has been generalized to higher even dimensions by Du, Iosevich, Ou, and Wang in \cite{Du3}. In odd dimensions, the best known dimensional threshold is due to Du and Zhang \cite{Du2} with $\dim_H(A)>\frac{d}{2}+\frac{d}{4d-2}$.

More recently, Ou and Taylor \cite{OS} studied a generalization for the setting of trees. More precisely, let $T$ be an arbitrary tree with $k+1$ vertices and $k$ edges. Assume that $V(T)=\{v_1, \ldots, v_{k+1}\}$, then the edge set of $T$ can be enumerated as follows:
\[\mathcal{E}(T)=\{ (v_{i_{1}}, v_{i_{2}}), (v_{i_3}, v_{i_4}), \ldots, (v_{i_{2k-1}}, v_{i_{2k}})\},\]
where $i_1\le i_3\le \cdots\le i_{2k-1}$, and $i_{2s}< i_{2t}$ if $s<t$ and $i_{2s-1}=i_{2t-1}$. 

The edge-length vector of $T$ is defined by 
\[(|v_{i_1}-v_{i_2}|, |v_{i_3}-v_{i_4}|, \ldots, |v_{i_{2k-1}}-v_{i_{2k}}|) \in \mathbb{R}^k.\]
Given a set $A\subset \mathbb{R}^2$, we say a tree $T'$, with vertices in $A$ given by $V(T')=\{x_1, \ldots, x_{k+1}\}$, is isomorphic to $T$ if there is a map $\varphi\colon V(T')\to V(T)$ such that $(x, y)$ is an edge of $T'$ iff $\left(\varphi(x), \varphi(y)\right)$ is an edge of $T$. For $x\in V(T')$ and $v\in V(T)$, we say $(T', x)$ is isomorphic to $(T, v)$ if $T$ is isomorphic to $T'$ and $\varphi(x)=v$. We say two isomorphic trees $T$ and $T'$ with $k$ edges are distinct if their edge-length vectors are distinct. Ou and Taylor proved the following theorem.
\begin{theorem}\label{O-T}
Given a tree $T$ with $k\ge 1$ edges and an arbitrary vertex $v$, let $A$ be a compact set in $\mathbb{R}^2$. If $\dim_H(A)>\frac{5}{4}$, then there exists $x\in A$ such that the set of edge-lengths of trees $(T', x)$ which are isomorphic to $(T, v)$ has positive Lebesgue measure. 
\end{theorem}
The main purpose of this paper is to extend this theorem to the discrete setting over finite fields and arbitrary fields. 

Let $\mathbb{F}$ denote an arbitrary field. When we deal with the special case that $\mathbb{F}$ is a finite field of order $q$, where $q$ is an odd prime power, we use the notation $\mathbb{F}_q$. The characteristic of $\mathbb{F}$ is denoted by $\mathtt{char}(\mathbb{F})$. We also use the  notation $X\gg Y$ to denote that there exists a constant $C>0$ such that $X\ge CY$.

In the setting of arbitrary fields $\mathbb{F}$, the distance between two points $x=(x_1, \ldots, x_d)$ and $y=(y_1, \ldots, y_d)$ in $\mathbb{F}^d$ is defined by 
\[||x-y||=(x_1-y_1)^2+\cdots+(x_d-y_d)^2.\]
The finite field analogue of the Falconer distance problem is known as the Erd\H{o}s-Falconer distance problem in the literature, which asks for the smallest exponent $\alpha$ such that for any $E\subset \mathbb{F}_q^d$, if $|E|\ge q^{\alpha}$, then the distance set $\Delta(E)$ covers a positive proportion of or all possible distances, where 
\[\Delta(E):=\{||x-y||\colon x, y\in E\}.\]
Iosevich and Rudnev proved in \cite{IR07} that $\alpha\le \frac{d+1}{2}$. They also showed that the exponent $\frac{d+1}{2}$ is sharp in odd dimensions. In even dimensions, it is also conjectured that the right exponent should be $\alpha=\frac{d}{2}$. Some partial progress on this conjecture can be found in \cite{CEHIK10, KPSV, KPV}. In two dimensions over prime fields, Murphy, Petridis, Pham, Rudnev, and Stevens \cite{murphy} recently proved that $\alpha\le \frac{5}{4}$, which is directly in line with Guth, Iosevich, Ou, and Wang's result for the Falconer distance problem. We also note that Murphy et al. actually proved a stronger statement, namely, if $|E|\ge p^{5/4}$, then there exist many points $x\in E$ such that $|\Delta_x(E)|\gg p$, where $\Delta_x(E)=\{||x-y||\colon y\in E\}$ is the set of pinned distances from $x$. Our first theorem is a generalization of this result in the form of Theorem \ref{O-T}.
\begin{theorem}\label{tree0}
Given a tree $T$ with $k\ge 1$ edges and an arbitrary vertex $v,$ let $E\subset \mathbb{F}_p^2$ with $|E|\ge 5\cdot 2^{16k} \cdot p^{5/4}$. There is a subset $E' \subset E$ with $|E'|\ge |E|/2^{24k+2}$ such that for any $x\in E'$, the number of distinct trees $(T', x)$ which are isomorphic to $(T, v)$ is at least $\gg p^k$. 
\end{theorem}
Our next theorem addresses the case of arbitrary fields. 
\begin{theorem}\label{tree1}
Given a tree $T$ with $k\ge 1$ edges and an arbitrary vertex $v$, let $E\subset \mathbb{F}^2$ with $|E|\gg 1$. If $\mathtt{char}(\mathbb{F})>0$, then we also need the restriction that $|E|\le \mathtt{char}(\mathbb{F})^{4/3}$. We assume in addition that any isotropic line contains at most $M=4^{-1}(16K)^{-(k-1)}(\lfloor |E|/8 \rfloor)^{\left(\frac{2}{3}\right)^{(k-1)}}$ points from $E$. Then there is a subset $E'\subset E$ with $|E'|\ge c(k)|E|^{\left(\frac{2}{3}\right)^{k}}$  such that for any $x\in E'$, the number of distinct trees $(T', x)$ which are isomorphic to $(T, v)$ is at least $\gg |E|^\frac{2k}{3}$, where $c(k)>0$ depends on $k,$ but is independent of $E$ and $q$.
\end{theorem}

If $\mathbb{F}$ is a specific field, says $\mathbb{R}$, then one would expect to have better bounds. For example, Guth and Katz \cite{GK10} proved that any set of $n$ points in $\mathbb{R}^2$ determines at least $n/\log n$ distinct distances, but for the pinned version, the best current record is due to Katz and Tardos \cite{KT} with the lower bound $n^{0.8641...}$. For $k=2$, the optimal result over the reals was given by Rudnev in \cite{Rud}. Namely, he showed that any set of $n$ points in $\mathbb{R}^2$ determines at least $n^2/\log^3 n$ distinct hinges. We conjecture that, in the form of Theorem \ref{tree1}, the optimal lower bound should be $\gtrsim n^k$.  For similar configurations, we refer the reader to  \cite{Ben, alex1, alex2, lyall} and references therein for more details.

\section{Preliminary: main lemmas}
Theorems \ref{tree0} and \ref{tree1} will be proved by induction, but the argument is complicated. The main difficulty arises from the property that a vertex might have many neighborhoods in the tree, which makes the iteration procedure become much harder. 

In inductive steps, we need strengthened versions of the distance result from Murphy et al. in \cite{murphy} for two sets which play a crucial role. Because of the peculiar effects of zero distances, we will frequently make use of the pinned nonzero distance set, $\Delta^*_x(E)$, the set of nonzero distances between a fixed point $x$ and points in $E$. More precisely, we define
$$\Delta^*_x(E):=\{||x-y||:y\in E, ||x-y||\neq 0\}.$$
We first recall the following result from \cite{murphy}. 
\begin{theorem}(\cite[Theorem 4]{murphy})\label{2.1}
Let $E$ be a subset in $\mathbb{F}_p^2$ with $5p^{5/4}\le |E| \le  p^{4/3}$. Let $\mathcal{T}^*(E, E, E)$, or in short $\mathcal{T}^*(E)$, be the number of triples $(a, b, c)\in E\times E\times E$ such that $||a-b||=||a-c||$ and $||b-c||\ne 0$. Then
\[\mathcal{T}^*(E)\le \frac{|E|^3}{p}+5p^{2/3}|E|^{5/3}+5p^{1/4}|E|^2.\] 
\end{theorem}
\begin{remark}
We note that the constant factor of the terms $p^{2/3}|E|^{5/3}$ and $p^{1/4}|E|^2$ is not stated precisely as such in \cite{murphy}, but it can be bounded by $5$ by carefully following the computations in the proof.
\end{remark}
Using this theorem, we have a result on distances between two different sets. 
\begin{lemma}\label{54}
Let $E$ and $F$ be two subsets in $\mathbb{F}_p^2$ with $5p^{5/4}\le |E|=|F| \le p^{4/3}$, then there exists a set $F'\subset F$ such that  $|F'|\ge \frac{1}{2^8}(|F|+1)$ and $|\Delta^*_x(E)|\gg p$ for any $x\in F'$.
\end{lemma}

\begin{proof}
Let $I$ be the number of isosceles triangles of the form $(x, a, b)\in F\times E\times E$ such that $||x-a||=||x-b||\ne 0$. We now set $n=|E|=|F|.$ It is clear that $I$ is bounded from above by $\mathcal{T}^*(E\cup F)$.  So, Theorem \ref{2.1} gives
\begin{equation}\label{isoEqn}
I\leq \frac{(2n)^3}{p}+20p^{2/3}n^{5/3}+20p^{1/4}n^2\le \frac{16n^3}{p}.
\end{equation}

Let $C_\delta$ be the circle centered at the origin of radius $\delta$.
By noticing that no point $x\in F$ can have a zero distance to more than $2p-1$ points in $E$, we see 
\begin{equation}\label{EF}
\frac{|E||F|}{2}\le |F|(|E|-2p+1)\le \sum_{x\in F}\sum_{\delta\in \Delta^*_x(E)}|E\cap (x+C_{\delta})|.
\end{equation}
Applying the Cauchy-Schwarz inequality to \eqref{EF}, we get
\[\frac{n^2}{2}=\frac{|E||F|}{2}\le \sum_{x\in F}\sum_{\delta\in \Delta^*_x(E)}|E\cap (x+C_{\delta})|\le \left(\sum_{x\in F}|\Delta^*_x(E)| \right)^{1/2}\cdot \left(\sum_{x\in F}\sum_{\delta\in \Delta^*_x(E)}|E\cap (x+C_\delta)|^2  \right)^{1/2}.\]
Noticing that the double sum on the right hand side is $I^\frac{1}{2},$ we apply \eqref{isoEqn} to get that this is bounded from above by
\[\left(\sum_{x\in F}|\Delta^*_x(E)| \right)^{1/2}\cdot \frac{4n^{3/2}}{p^{1/2}}.\]
Combining these calculations, we get that
\[\sum_{x\in F}|\Delta^*_x(E)|\ge \frac{np}{64}.\]
Let $F':=\{x\in F\colon |\Delta^*_x(E)|\ge p/129\}$. Now for contradiction, suppose that $|F'|\le |F|/129$. 

Then we would have 
\[\sum_{x\in F}|\Delta^*_x(E)|=\sum_{x\in F'}+\sum_{x\not\in F'}\le \frac{np}{129}+\frac{np}{129}<\frac{np}{64}.\]
This leads to a contradiction. Hence, 
\[|F'|\ge |F|/129 \ge (|F|+1) / 2^8.\]
\end{proof}
When the set is of large size, we have much more freedom when bounding the number of isosceles triangles. More precisely, in the following lemma, compared to Lemma \ref{54}, we do not require that $E$ and $F$ have the same size. 

\begin{lemma}\label{bigsize}
Let $E$ and $F$ be two subsets in $\mathbb{F}_p^2$ with  $|E|\ge 4p, |F|\ge 8p$, and $|F||E|^2\ge 64 p^4$, then there exists a set $F'\subset F$ such that  $|F'|\ge \frac{|F|}{24}$ and $|\Delta^*_x(E)|\gg p$ for any $x\in F'$.
\end{lemma}
\begin{proof}
In this proof, we first follow an argument of Hanson, Lund, and Roche-Newton for one set in \cite{hanson}. Let $\mathcal{T}^*(F, E, E)$ be the number of triples $(a, b, c)\in F\times E\times E$ such that $||a-b||=||a-c||$ and $||b-c||\ne 0$. We now show that
\[\mathcal{T}^*(F, E, E)\le \frac{|F||E|^2}{p}+p^{1/2}|F|^{1/2}\left(\frac{4|E|^4}{p^2}+10p|E|^2 \right)^{1/2}.\]
To prove this inequality, we first recall the definition of bisector energy from \cite{hanson}. For any two points $(a, b)\in E\times E$, we denote their bisector line by $B(a, b)$, the bisector energy associated to $E$ is defined by 
\[Q(E):=\{(a, b, c, d)\in E^4\colon B(a, b)=B(c, d), ||a-b||\ne 0\}.\]
Hanson et al. proved in their Theorem 3 that 
\[Q(E)\ll \frac{|E|^4}{p^2}+q|E|^2.\]
The explicit constant can be computed easily from their proof. Namely, in their Lemma 13, one would have the upper bound $\frac{|E|^4}{p}+2p^2|E|^2$ from Iosevich and Rudnev's paper \cite{IR07}. Substituting this upper bound to their Lemma 11 and Theorem 3 would give us 
\[Q(E)\le \frac{4|E|^4}{p^2}+10p|E|^2.\]
We now observe that the equation $||a-b||=||a-c||$ with $||b-c||\ne 0$ can be viewed as an incidence between $a\in F$ and the bisector line $B(a, c)$. So our problem is reduced to an incidence bound. Let $L$ be the multi-set of bisector lines determined by pairs of points of non-zero distances in $E$. The above observation tells us that 
\[\mathcal{T}^*(F, E, E)\le \mathcal{I}(F, L),\]
here $\mathcal{I}(F, L)$ means the number of incidences between points in $F$ and lines in $L$. 

In \cite[Lemma 9]{hanson}, it has been proved that
\[\mathcal{I}(F, L)\le \frac{|F||L|}{p}+p^{1/2}|F|^{1/2}\left(\sum_{\ell\in \overline{L}}m(\ell)^2\right)^{1/2},\]
where $\overline{L}$ is the set of distinct elements in $L$ and $m(\ell)$ is the multiplicity of $\ell$. 

We also note that 
\[\sum_{\ell}m(\ell)^2=Q(E).\]
In other words, we would obtain

\[\mathcal{T}^*(F, E, E)=\mathcal{I}(F, L)\le \frac{|F||E|^2}{p}+p^{1/2}|F|^{1/2}\left(\frac{4|E|^4}{p^2}+10p|E|^2 \right)^{1/2}.\]
So, if $|F| \geq 8p$ and $|F||E|^2 \geq 64 p^4$, we obtain 
\[\mathcal{T}^*(F, E, E)\le \frac{2|F||E|^2}{p}.\]
% \dung{
% We have \begin{align*}
% &\frac{|F||E|^2}{p}+p^{1/2}|F|^{1/2}\left(\frac{4|E|^4}{p^2}+10p|E|^2 \right)^{1/2} \leq \frac{2|F||E|^2}{p}\\
% \Leftrightarrow & p^{1/2}|F|^{1/2}\left(\frac{4|E|^4}{p^2}+10p|E|^2 \right)^{1/2} \leq \frac{|F||E|^2}{p} \\
% \Leftrightarrow & p^{3/2}|F|^{1/2}\left(\frac{4|E|^4}{p^2}+10p|E|^2 \right)^{1/2} \leq |F||E|^2 \\
% \Leftrightarrow &p^{3}|F|\left(\frac{4|E|^4}{p^2}+10p|E|^2 \right) \leq |F|^2|E|^4 \\
% \Leftrightarrow & p^{3}\left(\frac{4|E|^4}{p^2}+10p|E|^2 \right) \leq |F||E|^4
% \end{align*}
% Let $|F| |E|^2 \geq a p^4$. Then, from $|F| \geq 8p$, we have
% \begin{align*}
% & p^{3}\left(\frac{4|E|^4}{p^2}+10p|E|^2 \right)=4p|E|^4+10p^4|E|^2 \\
% \leq & \frac{|F|}{2} |E|^4+\frac{10}{a} (ap^4)|E|^2 \leq \frac{|F|}{2} |E|^4+\frac{10}{a} |F| |E|^4 \leq |F| |E|^4
% \end{align*}
% We only require that $\frac{10}{a}<\frac{1}{2}$ or $a>20$. We don't  need to choose $a=160$, we may choose $a=64$. In this case, in Theorems 2.5, 2.6 instead of $|F| \geq 160^{\frac{1}{3}} \cdot 2^{16k} \cdot p^{4/3}$ we assume $|F| \geq 4 \cdot 2^{16k} \cdot p^{4/3}$. In this lemma we can assume that $|F| \geq 8p$ and $|F| |E|^2 \geq 64p^4$ instead of $|F| \geq 16p$ and $|F| |E|^2 \geq 160 p^4$.
% }
We now run the same argument as in Lemma \ref{54} to conclude that there exists a set $F'\subset F$ with $|F'|\ge |F|/24$ and for all $x\in F'$, we have 
\[|\Delta_x^*(E)|\ge \frac{p}{24}.\]
\end{proof}
With the same approach, we obtain a similar result in the setting of arbitrary fields $\mathbb{F}$. Before stating it, we need a version of Theorem \ref{2.1} in $\mathbb{F}^2$.

\begin{theorem}(\cite[Theorem 1.4]{murphy})\label{tree-lm1}
Let $E$ be a set in $\mathbb{F}^2$. If $\mathtt{char}(\mathbb{F})>0$ is prime, then we need the restriction that $|E|\le \mathtt{char}(\mathbb{F})^{4/3}$. Let $\mathcal{T}^*(E)$ be the number of triples $(a, b, c)\in E\times E\times E$ such that $||a-b||=||a-c||$ and $||b-c||\ne 0$, then $\mathcal{T}^*(E)\le K|E|^{7/3}$, for some positive constant $K \geq 4$.
\end{theorem}

The next lemma is in line with Lemma \ref{54} with Theorem \ref{tree-lm1} in place of Theorem \ref{2.1}.

\begin{lemma}\label{arbi}
Let $E$ and $F$ be two finite sets in $\mathbb{F}^2$ with $|E|=|F|=n$. 
Assume that any isotropic line contains at most $M<n/4$ points from $E$. Then there exists a set $F'\subset F$ such that $|F'|\ge n^{2/3}/(8(K+1))$ and $|\Delta^*_x(E)|\gg n^{2/3}$ for any $x\in F'$, where the constant $K$ comes from Theorem \ref{tree-lm1}.
\end{lemma}
\begin{proof}
 Let $C_\delta$ be the circle of radius $\delta$ centered at the origin. Let $\Delta^*_x(E)$ be the set of non-zero distances from $x\in F$ to $E$. Because we have no more than $M$ points on any isotropic line, we have 
\[ \frac{n^2}{2}\le |F|(|E|-2M+1)\le \sum_{x\in F}\sum_{\delta\in \Delta^*_x(E)}|E\cap (x+C_{\delta})|.\]

Applying the Cauchy-Schwarz inequality as in the proof of Lemma \ref{54}, one has 
\[\frac{n^2}{2}\le \left(\sum_{x\in F}|\Delta^*_x(E)| \right)^{1/2}\cdot \left(\sum_{x\in F}\sum_{\delta\in \Delta^*_x(E)}|E\cap (x+C_\delta)|^2  \right)^{1/2}.\]
We now observe that 
\[\sum_{x\in F}\sum_{\delta\in \Delta^*_x(E)}|E\cap (x+C_{\delta})|^2\]
is the number of triples $(x, y, z)\in F\times E\times E$ such that $||x-y||=||x-z||\ne 0$. By applying Theorem \ref{tree-lm1} with $A=E\cup F$, and taking into account the degenerate cases where $y=z,$ the number of such triples is bounded by at most  $K|A|^{7/3}+|A|^2\le (K+1)|A|^{7/3}$. 

Thus, 
\[\sum_{x\in F}|\Delta^*_x(E)|\ge \frac{1}{4(K+1)}n^{5/3}.\]
Let  $F':=\{x\in F\colon |\Delta^*_x(E)|\ge \frac{1}{8(K+1)}n^{2/3}\}$. For contradiction, suppose that  $|F'|< \frac{1}{8(K+1)}n^{2/3}$. 

Then we would have 
  \[\sum_{x\in F}|\Delta^*_x(E)|=\sum_{x\in F'}+\sum_{x\not\in F'}< \frac{n^{5/3}}{8(K+1)}+\frac{n^{5/3}}{8(K+1)}<\frac{n^{5/3}}{4(K+1)}.\]

This leads to a contradiction. Hence, 
 \[|F'|\ge \frac{1}{8(K+1)}n^{2/3} . \]%% \geq \frac{1}{16K}(n^{2/3}+1). 
\end{proof}

\section{Proof of Theorem \ref{tree0}}
With lemmas in the previous section, we are ready to prove Theorem \ref{tree0}. We fall in two cases: sets of medium size and sets of large size.
\subsection{Sets of Medium size}
\begin{proof}[Proof of Theorem \ref{tree0} for sets of medium size]
We are going to prove the following stronger statement by induction: 
\begin{theorem}\label{techP}
Given a tree $T$ with $k$ edges and an arbitrary vertex $v$, let $E, F\subset \mathbb{F}_p^2$ with $5\cdot 2^{16k} \cdot p^{5/4}\le |E|=|F|\le p^{4/3}$. There is a subset $E'$ of $E$ with $|E'|\ge |E|/2^{8k}$ such that for any $x\in E'$, the number of distinct trees $(T', x)$ that are isomorphic to $(T, v)$ and have $V(T')\setminus \{x\}\subset F$ is at least $\gg p^{k}$. 
\end{theorem}
We prove the theorem by induction on $k$. The case $k=1$ follows from Lemma \ref{54}. For any $k>1,$ we now assume that the theorem holds for any tree with at most $k-1$ edges, then show that it also holds for the case of $k$ edges.

{\bf Case $1$:} The degree of $v$ in $T$ is one. Assume that the neighbor of $v$ in $T$ is $u$. Define $T_1:=T\setminus \{v\}$. By induction, we can find a subset $F_2'\subset F_2$ such that $|F_2'|\ge |F_2|/2^{8k-8}$ such that for any $y\in F_2'$, the number of distinct trees $(T_1', y)$, which are isomorphic to $(T_1, u)$ and $V(T_1')\setminus\{y\}\subset F_1$, is at least $\gg p^{k-1}$.  

We would like to apply Lemma \ref{54} to $E$ and $F_2',$ but they are not the same size, so we will break $E$ up into pieces of the appropriate size. Set  $s=\lfloor  |E|/|F_2'| \rfloor \geq 2$. Now, for $j=1, \dots, s,$ greedily select disjoint subsets $G_j\subset E,$ of size $|F_2'|.$  
We apply Lemma \ref{54} for $F_2'$ and $G_j$ for $j=1, \dots, s$ to conclude that there exists a subset $G_j'\subset G_j$ such that 
\begin{align}\label{ineq:G'-j-G-j}
|G_j'|\ge (|G_j|+1)/2^{8}
\end{align}
 and any $x\in G_j'$, we have $|\Delta^*_x(F_2')|\gg p$. Now set

$$E':=\bigcup_{j=1}^{s} G_j'.$$
Because the $G_j$ are disjoint, the size of the union is just the sum as $j$ ranges from $1$ to $s$. Combining \eqref{ineq:G'-j-G-j}, $|G_{j}|=|F^{\prime}_{2}|$, and $|E|=|F|$, gives that 
\begin{align*}
\left|E'\right| =&\sum\limits^{s}_{j=1} |G_j'| \ge \sum\limits^{s}_{j=1} \frac{|G_j|+1}{2^{8}} =\sum\limits^{s}_{j=1} \frac{|F'_{2}|+1}{2^{8}} \\
 \geq & s\left(\frac{|F_2|}{2^{8k}}+\frac{1}{2^8}\right) \geq s\left(\frac{|E|-1}{2^{8k+1}}+\frac{1}{2^8}\right) \\
\geq &\frac{s}{2}\left(\frac{|E|}{2^{8k}}\right) \geq \frac{|E|}{2^{8k}}.
\end{align*}
Every $x\in E'$, comes from some $G_j,$ so the number of distinct trees $(T', x)$ that are isomorphic to $(T, v)$ and $V(T')\setminus \{x\}\subset F$ is at least $\gg p^{k}$.

{\bf Case $2$:} The degree of $v$ in $T$ is more than one. Then we can partition the tree $T$ into two smaller trees, say $T_1$ and $T_2$ with $k_1$ and $k_2$ edges, respectively, sharing only the vertex $v$.

We will run the following argument for the sets $E_1$ and $E_2,$ separately. For a fixed choice of $i$, applying the induction hypothesis for the sets $E_i$ and $F_1$, we can say that there exists a subset $E_i'\subset E_i$ with $|E_i'|\ge |E_i|/2^{8k_1}$ such that for any $x\in E_i'$, the number of distinct trees $(T', x)$, which are isomorphic to $(T_1, v)$ and $V(T')\setminus \{x\}\subset F_1$, is at least $p^{k_1}$.

We now pick a subset $F_2'\subset F_2$ with $|F_2'|=|E_i'|.$ The induction hypothesis guarantees that there exists a subset $E_i''\subset E_i'$ with 
$$|E_i''|\geq|E_i'|/2^{8k_2}\geq|E_i|/2^{8(k_1+k_2)}=|E_i|/2^{8k},$$
such that for any $x\in E_i'',$ the number of distinct trees $(T'',x),$ which are isomorphic to $(T_2,v)$ and $V(T'')\setminus\{x\}\subset F_2'$ is at least $\gg p^{k_2}.$

It is clear that the number of distinct trees $(T', x)$, which are isomorphic to $(T, v)$ and $V(T')\setminus \{x\}\subset F$ is at least the product of the number of trees $(T', x)$ and $(T'', x)$ for any $x\in E_{i}''$, giving us at least $\gg p^{k_1+k_2}=p^k$. By running this argument for both $E_1$ and $E_2$, we complete the proof.

\end{proof}

\subsection{Sets of large size}
Using the same argument as in the case of sets of medium size and Lemma \ref{bigsize} in place of Lemma \ref{54}, the next theorem is obtained.
\begin{theorem}\label{techPP}
Given a tree $T$ with $k$ edges and an arbitrary vertex $v$, let $E, F\subset \mathbb{F}_p^2$ with  $|E|=|F|\ge 4\cdot  2^{16k}\cdot p^{4/3}$. There is a subset $E'$ of $E$ with $|E'|\ge |E|/2^{8k}$ such that for any $x\in E'$, the number of distinct trees $(T', x)$ that are isomorphic to $(T, v)$ and have $V(T')\setminus \{x\}\subset F$ is at least $\gg p^{k}$. 
\end{theorem}

\subsection{Conclusion}
Compared to the conclusion of Theorem \ref{tree0}, Theorems \ref{techP} and \ref{techPP} are stronger in the ranges $5\cdot 2^{16k}\cdot p^{5/4}\le |E|\le p^{4/3}$ and  $|E|\ge 4\cdot 2^{16k}\cdot p^{4/3}$. 

In the range,  $p^{4/3}\le |E|\le 4\cdot 2^{16k}\cdot p^{4/3}$, we simply take a subset of size $p^{4/3}$ and apply Theorem \ref{techP} to conclude that there exists a subset $E'\subset E$ with $|E'|\ge |E|/2^{24k+2}$ such that for any $x\in E'$, the number of distinct trees $(T', x)$ that are isomorphic to $(T, v)$ and have $V(T')\setminus \{x\}\subset F$ is at least $\gg p^{k}$. This completes the proof of Theorem \ref{tree0}.
\subsection{Remarks over arbitrary finite fields}
The argument in the proof of Theorem \ref{techPP} works not only over prime fields, but also over arbitrary finite fields. Therefore, in the setting of arbitrary finite fields $\mathbb{F}_q$, the following weaker result is achieved.

\begin{theorem}\label{techPPF_Q}
Given a tree $T$ with $k$ edges and an arbitrary vertex $v$, let $E, F\subset \mathbb{F}_q^2$ with $|E|=|F|\ge 4\cdot  2^{16k}\cdot q^{4/3}$. There is a subset $E'$ of $E$ with $|E'|\ge |E|/2^{8k}$ such that for any $x\in E'$, the number of distinct trees $(T', x)$ that are isomorphic to $(T, v)$ and have $V(T')\setminus \{x\}\subset F$ is at least $\gg q^{k}$. 
\end{theorem}

\section{Proof of Theorem \ref{tree1}}
The idea for the proof of Theorem \ref{tree1} is the same as that of Theorem \ref{tree0}, in particular, we will make use of  Lemma \ref{arbi} instead of Lemma \ref{54}. For small sets in an arbitrary field, we need to carefully run an analysis on the number of intersection points between isotropic lines and the sets $E$ and $F$. We now state and prove a more general statement.

\begin{theorem}\label{techArbi}
Given a tree $T$ with $k$ edges and an arbitrary vertex $v$, let $E, F\subset \mathbb{F}^2$ with $|E|=|F|$. If $\mathtt{char}(\mathbb{F})>0$, then we need the restriction that $|E|\le \mathtt{char}(\mathbb{F})^{4/3}$. We assume in addition that any isotropic line contains at most $M=4^{-1}(16K)^{-(k-1)}(\lfloor |F|/8 \rfloor)^{\left(\frac{2}{3}\right)^{(k-1)}}$ points from $E$ or $F$. There is a subset $F'$ of $F$ with $|F'|\geq (16K)^{-k}|F|^{\left(\frac{2}{3}\right)^{k}}$ such that for any $x\in F',$ the number of distinct trees $(T', x)$ that are isomorphic to $(T,v)$ with $V(T')\setminus \{x\}\subset E$ is at least  $\gg |E|^\frac{2k}{3}.$
\end{theorem}

We proceed by induction on the number of edges in the tree. Lemma \ref{arbi} yields the desired result for one edge. For a tree with $k>1,$ we assume that the statement of the theorem holds for any tree with at most $k-1$ edges. We will again select two subsets $E_1, E_2 \subset E$ both of size $\lfloor |E|/2 \rfloor,$ and similarly select $F_1, F_2\subset F.$

{\bf Case 1:} The degree of $v$ in $T$ is one. In this case, we denote the neighbor of $v$ in $T$ by $u$. Define $T_1:=T\setminus\{v\}.$ We now want to apply the induction hypothesis for two sets $F_1$ and $F_2$ for the tree $T_1$ of $k-1$ edges. But first, we need to verify that $F_1$ and $F_2$ satisfy the condition on the maximal number of points on each isotropic line with respect to their sizes. Specifically, we need to check that the bound of no more than $M$ points from $F$ on any isotropic line will guarantee that we do not have too many points from any isotropic line on either $F_i$, in terms of the size of the $F_i$, to apply the induction hypothesis. To be precise, we need that
\begin{equation}\label{MIH}
\frac{1}{4}\left(\frac{1}{16K}\right)^{k-2}\left( \left\lfloor \frac{|F_i|}{8} \right\rfloor \right)^{\left(\frac{2}{3}\right)^{k-2}}\ge M = \frac{1}{4}\left(\frac{1}{16K}\right)^{k-1}\left(\left\lfloor \frac{|F|}{8} \right\rfloor\right)^{\left(\frac{2}{3}\right)^{(k-1)}},
\end{equation}
which is true because
\[(16K)^{\left(\frac{3}{2}\right)^{k-2}}\left\lfloor\frac{|F_i|}{8}\right\rfloor\geq \left\lfloor\frac{|F|}{8}\right\rfloor^\frac{2}{3}.\]

Thus, we can find a subset $F_2'\subset F_2$ such that $|F_2'|\geq (16K)^{-(k-1)}|F_2|^{\left(\frac{2}{3}\right)^{k-1}}$ such that for any $y\in F_2'$, the number of distinct trees $(T_1',y)$ that are isomorphic to $(T_1,u)$ and have $V(T_1')\setminus\{y\}\subset F_1,$ is at least $\gg n^\frac{2(k-1)}{3}.$ %%and $V(T_1', u)$

We would like to apply Lemma \ref{arbi} to $E$ and $F_2',$ but they are not the same size, so as before we will break $E$ up into pieces of the appropriate size. Set $s=\lfloor  |F|/|F_2'| \rfloor \geq 2.$ Now, for $j=1, \dots, s,$ greedily select disjoint subsets $G_j\subset E,$ of size $|F_2'|.$ %Let $G_{s+1}$ be the set of leftover elements of $E$, so $|G_{s+1}|<|F_2'|$. 
Notice that the $G_j$ and $F_2'$ also satisfy the condition of having no more than $M$ points on any isotropic line, as they are subsets of $E$ and $F,$ respectively. Because $M\leq |F_2'| / 4 = |G_j|/4,$ we can apply Lemma \ref{arbi} for $F_2'$ and each $G_j$ to conclude that there exist a subsets $G_j'\subset G_j$ such that 
\begin{align}\label{est:G'-G-j}
|G_j'|\ge \frac{|G_j|^\frac{2}{3}}{8(K+1)}
\end{align}
and for any $x\in G_j'$, we have $|\Delta^*_x(F_2')|\gg |F_2'|^\frac{2}{3}$.  %\thang{Please double check, namely, 
%\[\frac{1}{4}|G_j|\ge 4^{-1}(16K)^{-(k-1)}(\lfloor |F|/8 \rfloor)^{\left(\frac{2}{3}\right)^{(k-1)}}.\] This can be reduced to 
%\[\frac{1}{4}\cdot \frac{1}{(16K)^{k-1}} |F_2|^{\left(\frac{2}{3}\right)^{k-1}}\ge 4^{-1}(16K)^{-(k-1)}(\lfloor |F|/8 \rfloor)^{\left(\frac{2}{3}\right)^{(k-1)}},\] where $|F_2|=\lfloor |F|/2 \rfloor$.}
Now set
\begin{align}\label{eq:E}
E':=\bigcup_{j=1}^{s} G_j'.
\end{align}
We are going to prove the following claim. 
\begin{align}\label{ine:E'-E}
|E'|\geq (16K)^{-k}|E|^{\left(\frac{2}{3}\right)^{k}}.
\end{align}

To prove \eqref{ine:E'-E}, we see from \eqref{eq:E}, \eqref{est:G'-G-j}, $s \geq 2$, and $|G_j|=|F_2'|$ that
\begin{align}\label{ineq:Eprime}
\nonumber
|E'| =&\sum\limits^{s}_{j=1} |G_j'| \ge \sum\limits^{s}_{j=1} \frac{|G_j|^\frac{2}{3}}{8(K+1)} =\sum\limits^{s}_{j=1} \frac{|F_2'|^\frac{2}{3}}{8(K+1)}
 \geq 2\frac{ \left((16K)^{-(k-1)}|F_2|^{\left(\frac{2}{3}\right)^{k-1}}\right)^\frac{2}{3}}{8(K+1)} \\
 \geq &(4K+4)^{-1}(16K)^{-\frac{2}{3}(k-1)}|F_2|^{\left(\frac{2}{3}\right)^{k}} \geq 4 (16K)^{-k}|F_2|^{\left(\frac{2}{3}\right)^{k}}. 
\end{align}
It follows from $k \geq 2$, $|F_{2}| \geq \frac{|E|-1}{2}$ that 
\begin{align}\label{ineq:F_2-E}
\nonumber
|F_2|^{\left(\frac{2}{3}\right)^{k}} \geq &\left(\frac{|E|-1}{2}\right)^{\left(\frac{2}{3}\right)^{k}} =|E|^{\left(\frac{2}{3}\right)^{k}} \left(\frac{1}{2}\right)^{\left(\frac{2}{3}\right)^{k}} \left(1-\frac{1}{|E|}\right)^{\left(\frac{2}{3}\right)^{k}} \\  
> &|E|^{\left(\frac{2}{3}\right)^{k}} \left(\frac{1}{2}\right)^{\left(\frac{2}{3}\right)^{k}} \cdot \frac{1}{2} 
\geq \frac{1}{4}|E|^{(\frac{2}{3})^{k}}.
\end{align}

Combining \eqref{ineq:Eprime} and \eqref{ineq:F_2-E}, we get \eqref{ine:E'-E}, as claimed.
This tells us that there exists a subset $|E'|\geq (16K)^{-k}|E|^{\left(\frac{2}{3}\right)^{k}}$ such that for any $x\in E',$ the number of distinct trees $(T', x)$ which are isomorphic to $(T,v)$ with $V(T')\setminus \{x\}\subset F$ is at least $\gg n^\frac{2k}{3}.$

{\bf Case $2$:} The degree of $v$ in $T$ is more than one. Then we can partition the tree $T$ into two smaller trees, say $T_1$ and $T_2$ with $k_1$ and $k_2$ edges, respectively, sharing only the vertex $v$. 

We will run the following argument for the sets $E_1$ and $E_2,$ separately. For a fixed choice of $i$, applying the induction hypothesis for the sets $E_i$ and $F_1$, we can say that there exists a subset $E_i'\subset E_i$ with $|E_i'|\ge (16K)^{-k_1} |E_i|^{\left(\frac{2}{3}\right)^{k_1}}$ such that for any $x\in E_i'$, the number of distinct trees $(T', x)$, which are isomorphic to $(T_1, v)$ and $V(T')\setminus \{x\}\subset F_1$, is at least $\gg n^\frac{2k_1}{3}$.

We now pick a subset $F_2'\subset F_2$ with $|F_2'|=|E_i'|.$ Again, we notice that since $E_i'$ and $F_2'$ are subsets of $E$ and $F,$ respectively, we are guaranteed that neither of them has more than $M$ points on any isotropic line, but we must again check the conditions of the induction hypothesis in order to apply it, as we did when verifying \eqref{MIH}. Specifically, we need to ensure that

\begin{equation}\label{MIH2}
\frac{1}{4}\left(\frac{1}{16K}\right)^{k_2-1}\left( \left\lfloor \frac{|F_i|}{8} \right\rfloor \right)^{\left(\frac{2}{3}\right)^{k_2-1}}\ge M=\frac{1}{4}\left(\frac{1}{16K}\right)^{k-1}\left(\left\lfloor \frac{|F|}{8} \right\rfloor\right)^{\left(\frac{2}{3}\right)^{k-1}},
\end{equation}
which is guaranteed as we recall $k=k_1+k_2$ and see
$$(16K)^{k_1+\left(\frac{3}{2}\right)^{k_2-1}} \left\lfloor\frac{|F_i|}{8}\right\rfloor\geq \left\lfloor\frac{|F_i|}{8}\right\rfloor \geq \left(\left\lfloor\frac{|F|}{8}\right\rfloor\right)^{\left(\frac{2}{3}\right)^{k_1}}.$$

The induction hypothesis tells us that there exists a subset $E_i''\subset E_i' \subset E$
such that for any $x\in E_i'',$ the number of distinct trees $(T'',x)$ that are isomorphic to $(T_2,v)$ and have $V(T'')\setminus\{x\}\subset F_2'$ is at least $\gg n^\frac{2k_2}{3}.$ Moreover, we also have
\begin{align*}
|E_i''|\ge &(16K)^{-k_2}|E_i'|^{\left(\frac{2}{3}\right)^{k_2}}\geq  (16K)^{-k_2}\left((16K)^{-k_1} |E_i|^{\left(\frac{2}{3}\right)^{k_1}} \right)^{\left(\frac{2}{3}\right) ^{k_2}} \\
\geq &  (16K)^{-(k_1 \cdot \left(\frac{2}{3}\right) ^{k_2}+k_2)}|E_i|^{\left(\frac{2}{3}\right)^{k_1} \left(\frac{2}{3}\right)^{k_2}} \\
\geq &4 (16K)^{-(k_1 +k_2)} \cdot |E_i|^{\left(\frac{2}{3}\right)^{k_1} \left(\frac{2}{3}\right)^{k_2}}
=4(16K)^{-k}|E_i|^{\left(\frac{2}{3}\right)^{k}}\geq (16K)^{-k}|E|^{\left(\frac{2}{3}\right)^{k}},
\end{align*}
where the last inequality can be explained by the same argument as in \eqref{ineq:F_2-E}.

It is clear that the number of distinct trees $(T', x)$ that are isomorphic to $(T, v)$ and have $V(T')\setminus \{x\}\subset F$ is at least the product of the number of trees $(T', x)$ and $(T'', x)$ for any $x\in E_{i}''$, giving us at least $\gg  n^\frac{2(k_1+k_2)}{3}=n^\frac{2k}{3}$. By running this argument for both $E_1$ and $E_2$, we complete the proof.

\section{Acknowledgements}
T. Pham would like to thank to the VIASM for the hospitality and for the excellent working condition.

\end{document}